\documentclass[11pt]{amsart}
\usepackage{amsmath,amsthm,amsfonts,amssymb,amscd,tikz}
\usepackage[all]{xy}

\newcommand{\C}{\mathbb{C}}

\newcommand{\N}{\mathbb{N}}

\newcommand{\R}{\mathbb{R}}

\newcommand{\cL}{{\mathcal L}}

\newcommand{\Hom}{\operatorname{Hom}}

\renewcommand{\ss}{\mathrm{ss}}
\newcommand{\rad}{\mathrm{rad}}

\newcommand{\Ext}{\mathrm{Ext}}

\newcommand{\modul}{\hbox{-mod}}

\newcommand{\nil}{\mathrm{nil}}

\newcommand{\Id}{\operatorname{Id}}

\newcommand{\Lie}{\operatorname{Lie}}

\newcommand{\id}{\operatorname{id}}

\newcommand{\bbR}{{\mathbb R}}

\newcommand{\gr}{\operatorname{\textrm{\rm gr}}}

\newtheorem{thm}{Theorem}[section]

\newtheorem{theorem}[thm]{Theorem}
\newtheorem{cor}[thm]{Corollary}
\newtheorem{corollary}[thm]{Corollary}
\newtheorem{prop}[thm]{Proposition}
\newtheorem{lemma}[thm]{Lemma}

\newtheorem*{claim*}{Claim}
\numberwithin{equation}{section}

\def\degp{\deg}

\begin{document}

\title[A note on Lie algebra cohomology]{A note on Lie algebra cohomology}

\author{Michael J. Larsen}
\address{Department of Mathematics,
Indiana University,
Bloomington, IN
47405,
U.S.A.}

\author{Valery A. Lunts}
\address{Department of Mathematics,
Indiana University,
Bloomington, IN
47405,
U.S.A.; HSE University, Russian Federation }

\thanks{ML was partially supported by NSF grant DMS-1702152.
VL was partially supported by Laboratory of Mirror Symmetry NRU HSE, RF Government grant, ag. No. 14.641.31.0001}


\begin{abstract}
Given a finite dimensional Lie algebra $L$ let $I$ be the augmentation ideal in the universal enveloping algebra $U(L)$. We study the conditions on $L$ under which the $\Ext$-groups $\Ext (k,k)$ for the trivial $L$-module $k$ are the same when computed in the category of all $U(L)$-modules or in the category of $I$-torsion $U(L)$-modules. We also prove that the Rees algebra $\oplus _{n\geq 0}I^n$ is Noetherian if and only if $L$ is nilpotent. An application to cohomology of equivariant sheaves is given.
\end{abstract}

\maketitle

\section{Introduction}

Let $L$ be a finite dimensional Lie algebra over a field $k$. Consider the universal enveloping algebra $U(L)$ with the augmentation ideal $I\subset U(L)$. Denote by $U(L)\modul$ the category of finitely generated left $U(L)$-modules and by $(U(L)\modul)_I\subset U(L)\modul$ the Serre subcategory of $I$-torsion modules. We have the obvious functor
\begin{equation}\label{functor}
\Phi _L:D^b((U(L)\modul)_I)\to D^b_I(U(L)\modul)
\end{equation}
where $D^b_I(U(L)\modul)\subset D^b(U(L)\modul)$ is the full triangulated subcategory consisting of complexes with $I$-torsion cohomology. In this paper we study the question:

\medskip

\noindent{\bf Question.} When is $\Phi _L$ an equivalence?

\medskip

The functor $\Phi _L$ being an equivalence means that the Ext-groups $\Ext ^i(k,k)$ for the trivial $L$-module $k$ are the same in the categories $U(L)\modul$ and $(U(L)\modul)_I$.

We answer this question in Theorem \ref{thm A} below.

Define inductively the decreasing sequence of ideals in $L$:
$$L_1=L,\quad L_n=[L,L_{n-1}]$$
and put $L_\infty =\bigcap _nL_n$. This is an ideal in $L$ such that the quotient Lie algebra $L_{\nil}:=L/L_\infty $ is nilpotent. We have $L_\infty =0$ if and only if $L$ is nilpotent.

For each $i$ the cohomology $H^i(L_\infty ,k)$ is naturally an $L_{\nil}$-module.
Denote by $H^{>0}(L_\infty ,k)$ the positive degree cohomology.

\begin{thm}\label{thm A} The functor $\Phi _L$ is an equivalence if and only if the $L_{\nil}$-module $H^{>0}(L_\infty ,k)$ has no subquotients isomorphic to the trivial module $k$. For example $\Phi _L$ is an equivalence if $L$ is nilpotent.
\end{thm}

We find it natural to approach Theorem \ref{thm A} by studying the graded  Rees algebra
$$U(L)^*:=\bigoplus _{n\geq 0}I^n=U(L)\oplus I\oplus I^2\oplus \cdots $$
It is easy to prove the following result.

\begin{prop} If the algebra $U(L)^*$ is graded left Noetherian, then the functor $\Phi _L$ is an equivalence.
\end{prop}

It is, however, not necessary for $U(L)^*$ to be graded left Noetherian in order for $\Phi _L$ to be an equivalence.

The next theorem may be of independent interest.

\begin{thm}\label{thm B} The algebra $U(L)^*$ is graded left Noetherian if and only if $L$ is nilpotent.
\end{thm}

In the last section of the paper we mention an application of Theorem~\ref{thm A} to the cohomology of quasi-coherent sheaves which are equivariant with respect to a unipotent group.

In this paper we consider only left modules, but all the results are also valid (with the same proofs) for right modules.

We fix a field $k$. All Lie algebras are finite dimensional over $k$. All associative rings are unital.
\vskip 6pt
\noindent{\it Acknowledgements.} We thank Grigory Papayanov for a useful discussion.

\section{A criterion for equivalence of categories}\label{sect criterion}
Let $R$ be an associative left Noetherian ring with a 2-sided ideal $I\subset R$. Let $M$ be a left $R$-module. An element $m\in M$ is called $I$-{\it torsion}, if $I^nm=0$ for some $n>0$. The collection of $I$-torsion elements in $M$ is an $R$-submodule, which we denote by $M_I$. We say that $M$ is torsion if $M_I=M$.

Let $R\modul$ denote the abelian category of finitely generated left $R$-modules and let $(R\modul)_I\subset R\modul$ be its full Serre subcategory of $I$-torsion modules. Let $C^b(R\modul)$  (resp. $C^b((R\modul)_I)$) be the category of bounded complexes over $R\modul$ (resp. over $(R\modul)_I$) and let  $C^b_I(R\modul)\subset C^b(R\modul)$ be the full subcategory of complexes whose cohomology groups are torsion.

In the bounded derived category $D^b(R\modul)$ consider the full subcategory $D^b_I(R\modul)$ of complexes with torsion cohomology groups.
We have the obvious functor
$$\Phi =\Phi _R :D^b((R\modul)_I)\to D^b_I(R\modul)$$

\begin{prop}\label{prop general} Assume that for every finitely generated left $R$-module $M$ there exists a submodule $N\subset M$ such that $N_I=0$ and $M/N$ is $I$-torsion.
Then the functor $\Phi$ is an equivalence.
\end{prop}

\begin{proof} Let $A^\bullet$ be an object of $C^b_I(R\modul)$. We claim there exists an object  $B^\bullet$ of $C^b((R\modul)_I)$ and a morphism of complexes $f:A^\bullet \to B^\bullet$ which is a quasi-isomorphism. Indeed, let
$$A^\bullet=0\to A^i\stackrel{d^i}{\longrightarrow} A^{i+1}\stackrel{d^{i+1}}{\longrightarrow} \cdots\stackrel{d^{n-1}}{\longrightarrow} A^n\to 0$$
and let $t$ be the lowest index such that $A^t_I\neq A^t$. By assumption there exists a submodule $P\subset A^t$ such that $P_I=0$ and $(A^t/P)_I=A^t/P$. We claim that $P\cap \ker d^t=0$. Indeed, since $H^t(A^\bullet)$ and $A^{t-1}$ are torsion, it follows that
$\ker d^t$ is torsion, so $P\cap \ker d^t=0$. Therefore, the complex $A^\bullet$ contains an acyclic subcomplex $\tilde{P}:=P\stackrel{\sim}{\to}d^t(P)$ and the components with index $\leq t$ of the quotient complex $A^\bullet /\tilde{P}$ are torsion. Iterating this process we find the required quasi-isomorphism $f:A^\bullet \to B^\bullet$. This shows that the functor $\Phi$ is essentially surjective.

For complexes $C^\bullet ,D^\bullet$ representing objects in $D^b((R\modul)_I)$, a morphism $\Phi (D^\bullet) \to \Phi (C^\bullet)$ is represented by a diagram of complexes $D^\bullet\to  A^\bullet \stackrel{s}{\leftarrow} C^\bullet$, where $A^\bullet \in C^b_I(R\modul)$ and $s$ is a quasi-isomorphism. The fact that the functor $\Phi$ is full and faithful now follows, since (as shown above) there exists a complex $B^\bullet \in C^b((R\modul)_I)$ and a morphism $f:A^\bullet \to B^\bullet$ of complexes that is a quasi-isomorphism.
\end{proof}

Consider now the graded Rees algebra
$$R^*:=\bigoplus _{n\geq 0}I^n=R\oplus I\oplus I^2\oplus \cdots $$

\begin{lemma} Assume that the algebra $R^*$ is graded left Noetherian (i.e. every graded left ideal is finitely generated). Then the assumption of Proposition \ref{prop general} holds: for any finitely generated left $R$-module $M$ there exists a submodule $N\subset M$ such that $N_I=0$ and $M/N=(M/N)_I$
\end{lemma}

\begin{proof} Let $M$ be a finitely generated $R$-module. If $M=M_I$, then we can take $N=0$. So assume that $M\neq M_I$, i.e. $I^sM\neq 0$ for all $s>0$. Consider the graded finitely generated $R^*$-module
$$\tilde{M}=M\oplus IM\oplus I^2M\oplus \cdots $$
and its graded submodule
$$P=M_I\oplus (IM\cap M_I)\oplus (I^2M\cap M_I)\oplus \cdots $$
By our assumption, $P$ is finitely generated; hence, there exists $n>0$ such that for all $m>0$
$$I^m(I^nM\cap M_I)=I^{n+m}M\cap M_I$$
As $I^n M\cap M_I$ is finitely generated and $I$-torsion, it is annihilated by $I^m$ for some $m$.  Thus,
$$(I^{m+n}M)_I =  I^{m+n}M\cap M_I = I^m (I^n M\cap M_I) = 0.$$
Putting
$N=I^{m+n}M$, we have $N_I=0$ and $(M/N)_I=M/N$.
 \end{proof}

\begin{corollary}\label{posit corollary} Assume that the algebra $R^*$ is graded left Noetherian. Then the functor $\Phi$ is an equivalence.
\end{corollary}

\section{When is the Rees algebra of a universal enveloping algebra Noetherian}

Let $L$ be a finite dimensional Lie algebra, $U(L)$ its universal enveloping algebra and $I\subset U(L)$ the augmentation ideal. As above, we consider the graded Rees algebra
$$U(L)^*=\bigoplus _{n\geq 0}I^n=U(L)\oplus I\oplus I^2\oplus \cdots $$
The main result of this section is the following theorem.

\begin{thm} \label{Noetherianity of rees algebra} The algebra $U(L)^*$ is graded left Noetherian if and only if the Lie algebra $L$ is nilpotent.
\end{thm}

Before proving the theorem we formulate a useful corollary.

\begin{cor} \label{spec cor for nilp} Let $L$ be a nilpotent Lie algebra. Then the functor
$$\Phi _{U(L)}:D^b((U(L)\modul)_I)\to D^b_I(U(L)\modul)$$
is an equivalence.
\end{cor}

\begin{proof} This follows from Theorem \ref{Noetherianity of rees algebra} and Corollary \ref{posit corollary}.
\end{proof}

\begin{proof} The proof of the theorem will take several steps and will occupy the rest of the section.

\subsection{Proof of the ``if" direction}
The universal enveloping algebra $U(L)$ has a standard increasing filtration which induces a similar filtration in the Rees algebra $U(L)^*$. We will prove that if $L$ is nilpotent, then the associated (double) graded algebra $\gr U(L)^*$ is commutative finitely generated, hence Noetherian.

So assume for now  that the Lie algebra $L$ is nilpotent.

We define $L^m$ for positive integers $m$ recursively;
$L^1 := L$, and
$$L^d := \sum_{j=1}^{d-1} [L^j,L^{d-j}]$$
for $m\ge 2$.  As $L$ is nilpotent, $L^d = \{0\}$ for all $d$ sufficiently large.
We choose an ordered basis $\{e_1,...,e_n\}$ of $L$ adapted to the decreasing filtration $L^\bullet$
in the sense that there exists a non-decreasing sequence $\nu_1\le \ldots\le \nu_n$ such that
$\{e_i\mid \nu_i\ge d\}$ spans $L^d$ for all $d$.
Thus, we can write
\begin{equation}
\label{key-ineq}
[e_i,e_j] = \sum_{\{k\mid \nu_k\ge \nu_i+\nu_j\}} c_k e_k.
\end{equation}

Write for short $U := U(L)$. We recall the standard increasing filtration $\{ U_r\}_{r\geq 0}$ of $U$  for which $U_r$ is the span of all products $x_1x_2\cdots x_r$,
where $x_i\in L\subset U$.
For every multi-index $a = (a_1,\ldots,a_n)$
of non-negative integers, we denote by $e^a$ the monomial $e_1^{a_1}\cdots e_n^{a_n}$.
Thus, $e^a\in U_{|a|}$, where $|a| := a_1 + \cdots + a_n$.  For every $x\in U$, we write $\degp x = 0$ if $x\in U_0$ and $\degp x = r$ if $x\in U_r\setminus U_{r-1}$ for $r\ge 1$.  We call this the \emph{degree} of $x$.

The PBW theorem implies that for all $r\in \N$ the image of the set
$$\{e^a: |a| = r\}\subset U_r$$
in $U_r/U_{r-1}$ is a basis.  Thus, the set of monomials $e^a$ as $a$ ranges over $\N^n$
is a basis of $U$, which we call the \emph{standard basis}, and $\degp \sum_a c_a e^a$ is the largest value of $|a|$ for which $c_a \neq 0$.
For each $r$, $U_r$ is spanned by $\{e^a: |a| \le r\}$.  Moreover, for all $x,y\in U$,
\begin{equation}
\label{PBW-sum}
\degp (x+y) \le \max(\degp x,\degp y).
\end{equation}
By the PBW theorem, if  $x$ and $y$ are non-zero elements of $U$, then
\begin{equation}
\label{PBW-prod}
\degp  xy = \degp x+\degp y.
\end{equation}

Let $I$ denote the augmentation ideal of $U$, or equivalently, the span of $e^a$ for all $a$ with $|a|\ge 1$.

\begin{lemma}
\label{I-power}
For all positive integers $m$, $I^m$ has basis
\begin{equation}
\label{power-basis}
\bigl\{e^a\bigm\vert \sum_i a_i \nu_i \ge m\bigr\}.
\end{equation}
\end{lemma}

\begin{proof}
This set is a subset of the standard basis, so it is linearly independent.  For $0<k<d$,  $[L^{d-k},L^k] \subset L^d$, so by induction on $d$,
every $e_i\in L^d$ lies in $I^d$.  Thus, the set (\ref{power-basis}) is contained in $I^m$.

We prove by induction that every element of $I^m$ is a linear combination of elements of (\ref{power-basis}), the case $m=1$ being trivial.
It suffices to prove that if $i_0,\ldots,i_k\in \{1,2,\ldots,n\}$, $i_1\le i_2\le \cdots\le i_k$, and
$$\nu _{i_1} + \cdots + \nu _{i_k} \ge m-1,$$
then $e_{i_0}e_{i_1}\cdots e_{i_k}$ lies in the span of (\ref{power-basis}) for $m$.  We prove more precisely
that for any sequence $i_0,\ldots,i_k\in \{1,\ldots,n\}$ which is non-decreasing with the possible exception of $i_0$,
the product $e_{i_0}e_{i_1}\cdots e_{i_k}$ is spanned by terms of the form $e^a$ where $|a| \le k+1$ and
\begin{equation}
\label{nilpotent-degree}
\sum_i a_i \nu_i \ge \nu _{i_0}+\cdots +\nu _{i_k}.
\end{equation}

We use double induction, first on $k$ and then on $i_0$.  The base case $k=0$ is trivial and for given $k$ there is nothing to prove if $i_0\le i_1$, so the base case
$i_0=1$ is trivial.  If $i_0 > i_1$, then writing
$$e_{i_0}e_{i_1}e_{i_2}\cdots e_{i_k} = e_{i_1} e_{i_0} e_{i_2}\cdots e_{i_k} + [e_{i_0},e_{i_1}] e_{i_2}\cdots e_{i_k},$$
we need only prove the claim for both summands on the right hand side.  By the induction hypothesis on $k$,
$e_{i_0}a_{i_2}\cdots e_{i_k}$ is a linear combination of $e^b$ for $|b| \le k$ and
$$\sum_i b_i \nu_i \ge   \nu_{i_0}+\nu_{i_2}\cdots +\nu_{i_k}.$$
By the induction hypothesis on $i_0$, $e_{i_1}$ times any such $e^b$ is a linear combination of terms $e^a$ with $|a|\le k+1$ and $a$ satisfying \eqref{nilpotent-degree}.

By (\ref{key-ineq}), we can write $[e_{i_0},e_{i_1}]$ as a linear combination of basis vectors $e_j$ with $\nu_j \ge \nu_{i_0}+\nu_{i_1}$, so by the
induction hypothesis on $k$, $[e_{i_0},e_{i_1}] e_{i_2}\cdots e_{i_k}$ is a linear combination of terms of the form $e^a$ where $|a|\le k$ and
\begin{align*}
\sum_i a_i \nu_i &\ge \nu_j + \nu_{i_2} + \cdots + \nu_{i_k} \\
&\ge \nu_{i_0} + \nu_{i_1} + \cdots \nu_{i_k}.
\end{align*}
\end{proof}

%
%
The filtration $\{U_r\}$ of the universal enveloping algebra $U=U(L)$ induces the filtration of the graded Rees algebra $U(L)^*$:
we filter each summand $I^m$ of $U(L)^*$ by the induced filtration given by $U_\bullet$:
$$F_rU(L)^* :=  \bigoplus_{m=0}^\infty (I^m\cap U_r)$$
%
%
%
We get
$$\gr U(L)^*  = \bigoplus_{r,m} (I^m\cap U_r) / (I^m\cap U_{r-1})$$
is a (doubly) graded commutative algebra. We denote
$$U(L)^*_{r,m}=(I^m\cap U_r) / (I^m\cap U_{r-1})$$
so that $\gr U(L)^*=\bigoplus _{r,m}U(L)^*_{r,m}$
By the PBW theorem, $U_1 = L$, so by Lemma~\ref{I-power}, $U(L)^*_{1,m}$ is naturally
identified with $L^m$.

\begin{prop}\label{graded is Noetherian}
The graded algebra $\gr U(L)^*$ is a commutative graded algebra generated by $U(L)^*_{1,m}$,
$m=0,1,\ldots,s$, where $L^{s+1} = \{0\}$.
\end{prop}

\begin{proof}
From Lemma~\ref{I-power}, we see that $I^m\cap U_r$ has basis
$$\bigl\{e^a\bigm\vert \sum_i a_i \nu_i \ge m, |a| \le r\bigr\},$$
so $U(L)^*_{r,m}$ has basis
$$\bigl\{e^a\bigm\vert \sum_i a_i \nu_i \ge m, |a| = r\bigr\}.$$
Notice that under the multiplication map
$$U(L)^*_{r,m}\otimes U(L)^*_{r',m'}\to U(L)^*_{r+r',m+m'}$$
we have $e^a\cdot e^b=e^{a+b}$.
To prove that the classes $U(L)^*_{1,m}$ generate, it suffices to note that the additive monoid
$$\bigl\{(a_1,\ldots,a_n,r,m)\in \N^{n+2}\bigm\vert \sum_i a_i \nu_i \ge m, |a| = r\bigr\}$$
is generated by the set
$$\bigl\{(\underbrace{0,\ldots,0}_{i-1},1,\underbrace{0,\ldots,0}_{n-i},1,m)\bigm\vert 1\le i\le n,\,0\le m\le \nu_i\bigr\}$$\
whose elements correspond to $e_i\in L^m = U(L)^*_{1,m} \subset L$.
\end{proof}

Now we recall a useful general result.
Let $R$ be an associative ring with increasing exhausting filtration
$$0=F_{-1}\subset F_0\subset F_1\subset \cdots =R$$
and consider the associated graded ring
$$\gr R=\gr _{F_\bullet}R=\bigoplus _{n\geq 0}F_{n}/F_{n-1}$$
If $x\in F_n\setminus F_{n-1}$ we say that $x$ has degree $n$ and denote by $\bar{x}$ its image in $F_n/F_{n-1}$.

\begin{lemma} \label{graded noeth then graded} Assume that $\gr R$ is graded left Noetherian. Then the ring $R$ is left Noetherian.
\end{lemma}

\begin{proof} Let $I\subset R$ be a left ideal. Then
$$\gr I:=\bigoplus _{n\geq 0}(I\cap F_{n})/(I\cap F_{n-1})$$
is a graded left ideal in $\gr R$. Let $\bar{x}_1,\ldots,\bar{x}_k$ be a set of homogeneous generators of the ideal $\gr I$; say $\deg (\bar{x}_i)=n_i$. Choose lifts $x_1,\ldots,x_k\in I$ of the $\bar{x}_i$'s and let $(x_1,\ldots,x_k)\subset I$ be the corresponding left ideal. We claim that $(x_1,\ldots,x_k)=I$. Indeed, let $x\in I$ be of degree $d$. If $d=0$, then clearly $x\in (x_1,\ldots,x_k)$. Otherwise, there exist $r_i\in R$ of degree $d-n_i$ such that
$$\bar{x}=\sum \bar{r}_i\bar{x}_i$$
It follows that $x-\sum r_ix_i\in I\cap F_{d-1}$, hence by induction on $d$, $x-\sum r_ix_i\in (x_1,\ldots,x_k)$.
\end{proof}

It follows from Proposition \ref{graded is Noetherian} and Lemma \ref{graded noeth then graded} that the Rees algebra $U(L)^*$ is left Noetherian if the Lie algebra is nilpotent. This proves the ``if" direction of Theorem \ref{Noetherianity of rees algebra}.

\subsection{Proof of the ``only if" direction}\label{only if}  For a Lie algebra $L$ we consider the lower central series $L_1=L,$ $L_n=[L,L_{n-1}]$. Thus
$$L=L_1\supset L_2\supset L_3\supset\cdots$$
is a nonincreasing sequence of ideals in $L$. We put
$$L_\infty :=\bigcap_nL_n,\quad L_{\nil}=L/L_\infty$$
The Lie algebra $L_{\nil}$ is nilpotent, and $L$ is nilpotent if and only if $L_\infty =0$.

We have the short exact sequence of Lie algebras
$$0\to L_\infty \to L\to L_{\nil}\to 0$$
This induces the surjection $\theta :U(L)\to U(L_{\nil})$ and $\ker \theta$ is the ideal $U(L)L_\infty U(L)$.
As before, let $I\subset U(L)$ be the augmentation ideal.

\begin{lemma} \label{nonempty intersection} We have the equality of ideals in $U(L)$:
$$\bigcap _nI^n=\ker \theta$$
\end{lemma}

\begin{proof} We have by construction $L_n\subset I^n$ for all $n$, hence $L_\infty \subset \bigcap_nI^n$ and so
$$\ker \theta =U(L)L_\infty U(L) \subset \bigcap_n I^n$$

To prove the opposite inclusion let $\overline{I}\subset U(L_{\nil})$ be the augmentation ideal, so that we have the surjection $\theta :I\to \overline{I}$. It suffices to prove that $\bigcap_n \overline{I}^n=0$. This follows from Lemma \ref{I-power}.
\end{proof}

Assume that the Lie algebra $L$ is not nilpotent, i.e. $L_\infty \neq 0$. Then Lemma \ref{nonempty intersection} implies that
$L_\infty \subset \bigcap_nI^n$. Fix $0\neq x\in L_\infty$ and consider the graded left ideal
$$J=\bigoplus _nU(L)x_n\subset U(L)^*$$
where $x_n$ denotes the copy of $x$ in $I^n$. We claim that $J$ is not finitely generated. Assume, on the contrary, that
$$J=\sum _iU(L)^*f_i$$
for a finite number of homogeneous elements $f_i\in U(L)x_{n_i}$. Choose $m>n_i$ for all $i$. We claim that
$$x_m\notin \sum _iU(L)^*f_i$$
Indeed, it suffices to notice that $x\notin Ix$: if $x=fx$, then $f=1$, (since $U(L)$ is a domain) and hence $f\notin I$.
This completes the proof of Theorem \ref{Noetherianity of rees algebra}.
\end{proof}

\section{Main theorem}
Let $L$ be a finite dimensional Lie algebra, $U(L)$ its universal enveloping algebra, $I\subset U(L)$ the augmentation ideal. As in Subsection \ref{only if}, consider the ideal
$$L_\infty =\bigcap_nL_n\subset L$$
and the quotient nilpotent Lie algebra $L_{\nil}=L/L_\infty$.

Each cohomology space $H^i (L_\infty ,k)$ is naturally a
$L_{\nil}$-module, so we have the Hochschild-Serre spectral sequence \cite{HS}:
\begin{equation}\label{hochschild serre}
E_2^{pq}=H^p(L_{\nil},H^q(L_\infty ,k))\Rightarrow H^{p+q}(L,k).
\end{equation}

\begin{theorem} \label{main} Let $L$ be a finite dimensional Lie algebra over a field $k$. The following conditions are equivalent:

(1) The natural functor
$$\Phi _L: D^b((U(L)\modul)_I) \to D^b_I(U(L)\modul)$$
is an equivalence.

(2) The natural map $H^\bullet (L_{\nil},k)\to H^\bullet (L,k)$ is an isomorphism.

(3) The positive degree cohomology $H^{>0}(L_\infty ,k)$  considered as an $L_{\nil}$-module  has no subquotients isomorphic to the trivial module $k$.
\end{theorem}

\begin{proof} We first notice that the 3 conditions in the theorem hold in case $L$ is nilpotent. Indeed, then $L_\infty =0$, so (2) and (3) hold trivially. Also (1) holds by Corollary \ref{spec cor for nilp}.

Let now $L$ be general. As in section \ref{only if} we consider the short exact sequence of Lie algebras
$$0\to L_\infty \to L\to L_{\nil}\to 0$$
and the induced surjection $\theta :U(L)\to U(L_{\nil})$ with the kernel $\ker \theta =\bigcap_nI^n$ (Lemma \ref{nonempty intersection}). Also, as in the proof of Lemma \ref{nonempty intersection}, denote by $\overline{I}$ the augmentation ideal in $U(L_{\nil})$. Lemma \ref{nonempty intersection} implies that any $U(L)$-module $M$ such that $M=M_I$ is actually a $U(L_{\nil})$-module (and $M=M_{\overline{I}}$). Hence the functor of restriction of scalars
$$\theta _*:U(L_{\nil})\modul\to U(L)\modul$$
induces the equivalence of categories
$$(U(L_{\nil})\modul)_{\overline{I}}\stackrel{\sim}{\to}  (U(L)\modul)_I$$
and therefore the equivalence of categories
\begin{equation}\label{vert equiv}
D^b((U(L_{\nil})\modul)_{\overline{I}})\stackrel{\sim}{\to}
D^b((U(L)\modul)_I)
\end{equation}

We have the commutative diagram of functors
\begin{equation}\label{commut diag funct}
\xymatrix{
D^b((U(L)\modul)_I) \ar[r]^(0.53){\Phi _L}& D^b_I(U(L)\modul)\\
D^b((U(L_{\nil})\modul)_{\overline{I}}) \ar[u] \ar[r]^(0.53){\Phi _{L_{\nil}}}& D^b_{\overline{I}}(U(L_{\nil})\modul)\ar[u]}
\end{equation}
As explained above the left vertical arrow is an equivalence. Also $\Phi _{L_{\nil}}$ is an equivalence (Corollary \ref{spec cor for nilp}). Hence $\Phi _L$ is an equivalence if and only if the functor
\begin{equation}\label{comp of ext}
\theta _*:D^b_{\overline{I}}(U(L_{\nil})\modul)\to
D^b_{I}(U(L)\modul)
\end{equation}
is an equivalence.
Every finitely generated $I$-torsion $U(L)$-module (resp. $\bar I$-torsion $U(L_{\nil})$-module) is a finite dimensional  $k$-vector space
on which $U(L)$ (resp. $U(L_{\nil})$) acts nilpotently, so by Engel's theorem, it admits a stable flag with trivial $1$-dimensional quotients.
As triangulated categories, therefore,
both sides  of \eqref{comp of ext} are generated by the trivial module $k$, and so the functor  in
\eqref{comp of ext} is an equivalence if and only if the natural map
\begin{equation}\label{final comp of ext}
\theta _*:\Ext ^\bullet _{U(L_{\nil})}(k,k)\to
\Ext ^\bullet _{U(L)}(k,k)
\end{equation}
is an isomorphism. This proves the equivalence of conditions (1) and (2) in the theorem. It remains to prove the equivalence of (2) and (3).

First we prove a lemma.
Let $\cL$ be a nilpotent Lie algebra. By a theorem of Lie every simple $\cL$-module $M$ is one dimensional, hence it corresponds to an additive character
$$\chi _M:\cL/[\cL,\cL]\to k$$

\begin{lemma} \label{nontrivial trivial} Let $M$ be a simple nontrivial $\cL$-module. Then
$$H^\bullet (\cL,M)=0$$
\end{lemma}

\begin{proof} The character $\chi _M$ gives a short exact sequence of nilpotent Lie algebras
\begin{equation}\label{short exact nilp}
0\to \cL '\to \cL \stackrel{\chi _M}{\to }\overline{\cL}\to 0
\end{equation}
with $\dim \overline{\cL}=1$, where $\cL '$ acts trivially on $M$.

Recall \cite[(23.1)]{CE} the standard complex $C^\bullet (\cL ' ,M)$, which computes the cohomology $H^\bullet (\cL ',M)$.  With our normalization, it is
\begin{equation}\label{hs-complex}
0\to M=C^0(\cL ',M)\stackrel{\partial _0}{\to} C^1(\cL ' ,M)\stackrel{\partial _1}{\to } \cdots,
\end{equation}
where $C^n(\cL ',M):=\Hom _k(\bigwedge ^n\cL ' ,M)$ and
\begin{align*}
\partial _n(f)(x_0\wedge&\cdots \wedge x_n)  =  \sum _{i=0}^n(-1)^ix_if(x_0\wedge \cdots \wedge \hat{x}_i\wedge \cdots \wedge x_n)\\
& + \sum _{p<q}(-1)^{p+q}f([x_p,x_q]\wedge x_0\wedge \cdots \wedge \hat{x}_p\wedge \cdots \wedge \hat{x} _q\wedge \cdots \wedge x_n)
\end{align*}
The $\cL $-action on $\cL '$ and $M$ extends to an action on the complex $C^\bullet (\cL ',M)$ by the formula
\begin{equation} \label{action on complex}
\begin{split}
(xf)(x_1\wedge \cdots \wedge x_n) & =  xf(x_1\wedge \cdots \wedge x_n)\\
                               &  -\sum _{i=1}^nf(x_1\wedge \cdots \wedge x_{i-1}\wedge [x,x_i]\wedge x_{i+1}\wedge \cdots \wedge x_n)
\end{split}
\end{equation}
This induces the $\cL$-action on the cohomology $H^\bullet (\cL ',M)$
which is trivial on $\cL '$ and hence gives the required  $\overline{L}$-action on $H^\bullet (\cL ',M)$.

We now take a closer look at the $\cL$-action \eqref{action on complex} on $C^n(\cL ',M)$. Since the $\cL$-action on the Lie algebra $\cL '$ is nilpotent there exists a basis of $\cL '$ such that for every $x\in \cL $ the matrix of the operator $[x,-]$ in this basis is strictly lower triangular. Then it follows from the formula
\eqref{action on complex} that there exists a basis for $C^n(\cL ',M)$ such that the action of every $x\in \cL$ is given by a lower triangular matrix with all diagonal entries being $\chi _M(x)$.

Therefore any $\overline{x}\in \overline{\cL}$ acts on the cohomology $H^\bullet (\cL ',M)$ by an operator whose   characteristic polynomial is a power of $(t-\chi _M(\overline{x}))$.
Since we assume that $\chi _M(\overline{x})\neq 0$ for $\overline{x}\neq 0$, it easily follows that
$$\Ext ^\bullet _{U(\overline{\cL})}(k,H^\bullet (\cL ',M))=H^\bullet (\overline{\cL},H^\bullet (\cL ',M))=0$$
Now the Hochschild-Serre spectral sequence
$$E_2^{pq}=H^p(\overline{\cL }, H^q(\cL ',M))\Rightarrow H^{p+q}(\cL ,M)$$
implies that $H^\bullet (\cL ,M)=0$, which proves the lemma.
\end{proof}

We now return to the proof of the equivalence of conditions (2) and (3) in the theorem. The Hochschild-Serre spectral sequence \eqref{hochschild serre}
has $E_2$ page
\begin{equation}\label{e2 page}
\xymatrix{
E_2^{01}\ar[drr] & E_2^{11}\ar[drr]& \cdots&\cdots\\
E_2^{00} & E_2^{10} & E_2^{20} & \cdots \\
}
\end{equation}
We have $H^0(L_\infty ,k)=k$ -- the trivial $L_{\nil}$-module and the graded space $\Ext ^\bullet _{U(L_{\nil})}(k,k)$ identifies naturally with the bottom row of this spectral sequence. The map $\theta ^* :\Ext ^\bullet _{U(L_{\nil})}(k,k)\to
\Ext ^\bullet _{U(L)}(k,k)$ then coincides with the projection
$$\Ext ^\bullet _{U(L_{\nil})}(k,k)=H^\bullet (L_{\nil},H^0(L_\infty ,k))\to H^\bullet (L,k)$$

Notice that by a theorem of Lie every finite dimensional $L_{\nil}$-module has a filtration with 1-dimensional subquotients.

Assume that the condition (3) holds, i.e. the $L_{\nil}$-module $H^{>0}(L_{\infty},k)$ has no subquotients isomorphic to the trivial module $k$. Then all the  $1$-dimensional subquotients of the $L_{\nil}$-module $H^{>0}(L_{\infty},k)$ are nontrivial. In this case it follows from Lemma \ref{nontrivial trivial} that only the bottom row of the spectral sequence \eqref{e2 page} is nonzero. Therefore the natural map $H^\bullet (L_{\nil},k)\to H^\bullet (L,k)$ is an isomorphism, i.e. the condition (2) of the theorem holds.

Assume, conversely, that condition (2) holds.  Let $d$ be the maximal integer such that $H^d(L_{\nil},k)\neq 0$. Again using Lie's theorem and Lemma \ref{nontrivial trivial} it follows that
$H^{>d}(L_{\nil},N)=0$ for any finite dimensional $L_{\nil}$-module $N$.  If $k$ is a subquotient of $N$, then by Engel's theorem, it is also a quotient of $N$.
It follows, therefore, that if $N$ has the trivial module as a  subquotient, then $H^d(L_{\nil},N)\neq 0$. So if for some $i>0$ the $L_{\nil}$-module $H^i(L_\infty ,k)$ has the trivial submodule $k$ as a subquotient then
$E_2^{d,i}=H^d(L_{\nil}, H^i(L_\infty ,k))$ is nonzero, which means that it survives in $H^{i+d}(L,k)$. This is a contradiction and finishes the proof of the theorem.
\end{proof}

\subsection{Some examples}

(A) Consider the 2-dimensional Lie algebra with basis $x,y$ and the relation $[x,y]=y$. This Lie algebra is solvable but not nilpotent, $L_\infty=ky$. The standard complex, which computes the cohomology $H^\bullet (L_\infty,k)$ has terms in degrees $0$ and $1$ and zero differential
$$k\stackrel{0}{\to }\Hom _k(ky,k)$$
The element $x\in L_{\nil}$ acts on the space $\Hom _k(ky,k)=H^1(L_\infty,k)$ as minus the identity (see formula \eqref{action on complex}), hence the condition (3) of Theorem \ref{main} is satisfied.

(B) This is a generalization of example (A) above: assume that the $L_{\nil}$-module $\bigwedge ^iL_\infty$ has no subquotients isomorphic to the trivial module $k$ if $i>0$. Then the condition (3) of Theorem \ref{main} holds. For example this is the case when $L$ is the Lie algebra of upper-triangular matrices. Then $L_\infty$ is the ideal of strictly triangular matrices and $L_{\nil}$ is the abelian quotient.

(C) However, there exists solvable algebras $L$ for which condition (3) does not hold.  See Proposition~\ref{solvable-but-non-L} below.

(D) It may happen that condition (3) holds even though $\bigwedge^{>0} L_\infty$ admits $k$ as an $L_{\nil}$-subquotient.  See Proposition~\ref{amazing-L} below.

For the next two examples, we assume that $k$ has characteristic zero.

(E) Assume that $L_\infty \neq 0$ is semi-simple.  If $k=\R$ and $L_\infty$ is compact, the cohomology $H^\bullet (L_\infty ,k)$ is isomorphic to the cohomology of any compact Lie group with the Lie algebra $L_\infty$  \cite[Theorem~15.2]{CE}, so $H^{>0}(L_\infty ,k)\neq 0$.  Since every compact semisimple Lie algebra has a compact real form, $H^{>0}(L_\infty ,k)\neq 0$ when $k=\C$ and $L_\infty$ is semisimple, the same statement follows for every Lie algebra over any field $k$ of characteristic zero.
However for any $x\in L_{\nil}$ the operator $[x,-]$ on of $L_\infty$ is a derivation, so is inner.
Therefore the action of $L_{\nil}$ of the cohomology is trivial and so the condition (3) of Theorem \ref{main} fails.

(F) This is a generalization of example (E) above. We formulate it as a proposition.

\begin{prop}
Assume that the equivalent conditions of Theorem \ref{main} are satisfied and the characteristic of $k$ is zero.  Then the algebra $L_\infty $ is solvable.
\end{prop}

\begin{proof} By the general theory the Lie algebra $L_\infty $ has a maximal solvable ideal (the radical) $L_\infty ^{\rad}$ such that
$$L_\infty ^{\ss}:= L_\infty /L_\infty ^{\rad}$$
is semi-simple. Put $\mathfrak{g}:=L_\infty ^{\ss}$. We need to prove that
$\mathfrak{g}=0$.

\begin{lemma} In the above notation the natural map
$$H^\bullet (\mathfrak{g},k)=\Ext^\bullet  _{U(\mathfrak{g})}(k,k)\to
\Ext ^\bullet _{U(L_\infty)}(k,k)=H^\bullet (L_\infty ,k)$$
is injective.
\end{lemma}

\begin{proof} By the Levi theorem we know that the surjection of Lie algebras $p:L_\infty \to \mathfrak{g}$ has a splitting $s:\mathfrak{g}\to L_\infty$. The corresponding surjection of universal enveloping algebras $p:U(L_\infty )\to U(\mathfrak{g})$ induces the pair of adjoint functors $(Lp*,p_*)$ -- the extension and restriction of scalars between the derived categories $D^b(U(L_\infty )\modul)$ and $D^b(U(\mathfrak{g})\modul)$, where $A\modul$ is the category of (all) left modules over an associative ring $A$. It suffices to prove that the adjunction morphism of functors
$$\Id_{D^b(U(\mathfrak{g})\modul)}\to Lp^*\cdot p_*$$
has a left inverse, i.e. $\Id_{D^b(U(\mathfrak{g})\modul)}$ is a direct summand of the functor $Lp^*\cdot p_*$.

Let $S^\bullet$ be an object in $D^b(U(\mathfrak{g})\modul)$.
We may assume that $S^\bullet$ consists of projective $U(\mathfrak{g})$-modules.
Let us construct a special (functorial) projective resolution of $p_*S^\bullet$.
The morphism $s:\mathfrak{g}\to L_\infty$ gives a homomorphism $s:U(\mathfrak{g})\to U(L_\infty)$ such that $p\cdot s=\id$. So we may consider $U(L_\infty)$ as a (free) right $U(\mathfrak{g})$-module via the homomorphism $s$. Consider the obvious short exact sequence of complexes of $U(L_\infty)$-modules
 $$0\to K^\bullet \to U(L_\infty)\otimes _{U(\mathfrak{g})}S^\bullet \to p_*S^\bullet \to 0$$
Now we repeat this procedure with $K^\bullet $ instead of $S^\bullet$ (by first considering $K^\bullet$ as a complex of $U(\mathfrak{g})$-modules via the map $s$) and so on. Eventually we obtain the complex of (complexes of projective) $U(L_\infty)$-modules
$$P^\bullet:= \cdots\stackrel{\partial _1}{\to} U(L_\infty)\otimes _{U(\mathfrak{g})}K^\bullet \stackrel{\partial _0}{\to} U(L_\infty )\otimes _{U(\mathfrak{g})}S^\bullet $$
which is a resolution of $p_*S^\bullet$ and hence
$$Lp^*\cdot p_*S^\bullet =U(\mathfrak{g})\otimes _{U(L_\infty)}P^\bullet$$

Note that by construction of $P^\bullet$, the map $U(\mathfrak{g})\otimes _{U(L_\infty)}\partial _0$ is zero.
Hence $U(\mathfrak{g})\otimes _{U(L_\infty)}(U(L_\infty)\otimes
_{U(\mathfrak{g})} S^\bullet)=S^\bullet$ is a direct summand of $Lp^*\cdot p_*S^\bullet$, which proves the lemma.
\end{proof}
Now the assertion of the proposition follows from the example (E) above. Indeed, if $\mathfrak{g}\neq 0$, then $H^{>0}(\mathfrak{g},k)\neq 0$. As explained in example (E) the $L_{\nil}$ action on $H^{>0}(\mathfrak{g},k)\neq 0$ is trivial, so by the above lemma the space $H^{>0}(L_\infty ,k)$ contains a nonzero $L_{\nil}$-submodule, which is trivial. This contradicts condition (3) of Theorem \ref{main}.
\end{proof}

\begin{prop}
\label{solvable-but-non-L}
There exist solvable Lie algebras $L$ not satisfying the equivalent conditions of Theorem~\ref{main}.
\end{prop}

\begin{proof}
For a nilpotent Lie algebra $N$ acting on the trivial module $k$, the top differential,  $C^{\dim N-1}(N,k)\to C^{\dim N}(N,k)$, of the Chevalley-Eilenberg complex is always zero.  Therefore, if $L_\infty$ is nilpotent
and $L_{\nil}$ acts trivially on $\bigwedge^{\dim L_\infty} L_\infty$, then condition (3) of Theorem~\ref{main} is violated.

Let $L$ be the subalgebra of upper-triangular $3\times 3$ matrices such that the upper left and lower right entries are the same.  Thus $L_\infty$ consists of strictly upper triangular matrices, and its center $Z$ consists of matrices which are zero except possibly in the upper right entry.  The commutator map $L_\infty/Z\times L_\infty/Z\to Z$ respects the action of $L_{\nil}$, which is trivial on $Z$.
This implies that $\bigwedge^2 (L_\infty/Z)\cong k$ as $L_{\nil}$-module, and therefore $\bigwedge^3 L_\infty\cong k$ as $L_{\nil}$-module.
\end{proof}

\begin{prop}
\label{amazing-L}
The conditions of Theorem~\ref{main} are strictly weaker than the condition
that $(\bigwedge^{>0}L^*_\infty)$ has a non-trivial $L_{\nil}$-invariant subquotient.
\end{prop}
\begin{proof}
As $H^{>0}(L_\infty,k)$ is a subquotient of $\bigwedge^{>0} L_\infty^*$,
if the former has a non-trivial $L_{\nil}$-invariant subquotient, the latter does as well.

We show that converse does not hold by exhibiting a case in which $L_{\nil}$ acts semisimply on
$\bigwedge^\bullet L_\infty^*$ and therefore on every $L_{\nil}$-stable subquotient and for which
$$\dim (\bigwedge\nolimits^\bullet L_\infty^*)^{L_{\nil}} > 1 = \dim H^\bullet(L_\infty,k)^{L_{\nil}}.$$

The free Lie algebra on two generators $x$ and $y$ admits a unique bigrading for which $x$ and $y$ have bidegree $(1,0)$ and $(0,1)$ respectively.
Let $M$ be the quotient of this algebra by the graded ideal generated by all elements of total degree $\ge 4$ and also $[[x,y],y]$.
Then $M$ has basis: $x, y, z, w$ of bidegree $(1,0)$, $(0,1)$, $(1,1)$, and $(2,1)$ respectively, satisfying the following relations:
$$[x,y]=z,\,[x,z]=w,\,[x,w]=[y,z]=[y,w]=[z,w]=0$$
(see \cite[II, \S2, no.\ 11, Th\'eor\`eme 1]{Bourbaki} and the computation of the Hall set for $2$ generators given at the end of no.\ 10.)

We define $t$ to be the derivation which acts on the bidegree $(a,b)$ part of $M$ by $2a-3b$.
Let $L := M\oplus k t$ denote the semi-direct sum, so
$$[t,x] = 2x,\,[t,y] = -3y,\,[t,z] = -z,\,[t,w] = w.$$
We confirm that $[L,L] = [L,M] = M$, so $L_\infty = M$, and $L_{\nil}$ is the $1$-dimensional algebra spanned by the class of $t$.

Next, we consider the Chevalley-Eilenberg complex of $M$.  The underlying graded space is $\bigwedge^\bullet M^*$, which is spanned by wedge products of the dual basis
$x^*,y^*,z^*,w^*$ of $M$.  The differential is given by
$$\delta(x^*) = 0,\,\delta(y^*) = 0,\,\delta(z^*) = y^* \wedge x^*,\,\delta(w^*) = z^*\wedge x^*.$$
The bigrading on $M$ induces a bigrading on $\bigwedge^\bullet M^*$, and $\delta$ preserves bidegree.  The degree $(3,2)$-part of $\bigwedge^\bullet M^*$ is spanned by $z^*\wedge w^*$ and $x^*\wedge y^*\wedge w^*$.  As
$$\delta(z^*\wedge w^*) = -x^*\wedge y^*\wedge z^*,$$
the degree $(3,2)$-part of $H^*(L_\infty,k)$ is zero.  On the other hand, the $t$-invariant part of
$H^*(L_\infty,k)$ is the sum of the $(3n,2n)$-part over all integers $n$.

Now $H^*(L_\infty,k)$ is a subquotient of $\bigwedge^\bullet M^*$, and the latter has non-trivial degree $(3n,2n)$-part only for $n=0,1$.  Thus, $\dim (\bigwedge^\bullet L_\infty)^{L_{\nil}} = 3$ but
$$\dim H^\bullet(L_\infty,k)^{L_{\nil}} = 1.$$
\end{proof}

\section{An application}

Let $k$ be an algebraically closed field of characteristic zero. Let $V$ be a linear unipotent algebraic group over $k$, $L=\Lie V$ the corresponding nilpotent Lie algebra.

Denote by $V\text{-Mod}$ the abelian category of rational representations of $V$. Recall that an object of $V\modul$ is by definition a $V$-module $M$ which is a union of finite dimensional submodules $M_i$, such that the $V$-action on $M_i$ comes from a homomorphism of $k$-algebraic groups $V\to GL(M_i)$. In particular every element of $V$ acts on $M_i$ via a unipotent operator.

Notice that we have a natural equivalence of abelian categories
$$\log :V\text{-Mod} \to (U(L)\text{-Mod})_I$$
where $(U(L)\text{-Mod})_I$ is the abelian category of (all) $U(L)$-modules which are $I$-torsion.

This induces the equivalence of derived categories
\begin{equation}\label{derived log} \log :D^b(V\text{-Mod}) \to D^b((U(L)\text{-Mod})_I)
\end{equation}

Recall that for $M\in V\modul$ its cohomology is by definition
$$H^\bullet _V (M):=\Ext ^\bullet _{V\text{-Mod}}(k,M)$$
where $k$ is the trivial rational $V$-module.

\begin{cor} \label{appl to coh} For any $M\in V\text{-Mod}$ we have the isomorphism
\begin{equation}\label{equality of cohomology} H^\bullet _V(M)\simeq H^\bullet (L,\log(M))
\end{equation}
In particular, the cohomology $H^\bullet _V(M)$ can be computed using the
standard complex for the Lie algebra $L$.
\end{cor}

\begin{proof} The equivalence \eqref{derived log} implies the isomorphism
\begin{equation}\label{cor for ext} \Ext ^\bullet _{V\text{-Mod}} (k,M)=\Ext ^\bullet _{(U(L)\text{-Mod})_I}(k,\log(M))\end{equation}
The module $M$ is a direct limit (union) of its finite dimensional submodules.
The cohomology on both sides of \eqref{equality of cohomology} commutes with direct limits, hence
we may assume that $\dim _kM<\infty$ and so the $\log(M)\in U(L)\modul$.
Using the standard methods one can show that
$$\Ext ^\bullet _{(U(L)\text{-Mod})_I}(k,\log(M))=\Ext ^\bullet _{(U(L)\modul)_I}(k,\log(M))$$
Finally, Corollary \ref{spec cor for nilp} implies the isomorphism
$$\Ext ^\bullet _{(U(L)\modul)_I}(k,\log(M))=\Ext ^\bullet _{U(L)\modul}(k,\log(M))$$
which proves the corollary.
\end{proof}

Let $X$ be a $k$-scheme with an action of the group $V$. For a $V$-equivariant quasi-coherent sheaf $F$, its cohomology can be computed as
$$H^\bullet _V(X,F)=\Ext ^\bullet _{V\text{-Mod}}(k,\bbR\Gamma
(X, F)),$$
and sometimes one wants to know that the Ext-space
$\Ext ^\bullet _{V\text{-Mod}}(k,-)$ can be computed using the standard complex for the Lie algebra $L$ (by Corollary \ref{appl to coh}). This fact was used, for example, in the key computation on p.~8 of \cite{Te}.

\end{document}